\documentclass[11pt, reqno]{amsart}

\usepackage{amsmath, amsthm, amscd, amsfonts, amssymb, graphicx, color}
\usepackage[bookmarksnumbered, colorlinks, plainpages]{hyperref}

\newtheorem{theorem}{Theorem}[section]
\newtheorem{lemma}[theorem]{Lemma}

\newtheorem{corollary}[theorem]{Corollary}
\theoremstyle{definition}

\theoremstyle{remark}
\newtheorem{remark}[theorem]{Remark}
\numberwithin{equation}{section}
\DeclareMathOperator\pp{\perp}
\def\np#1{\sideset{ ^{#1\!\!}}{}\pp}

\begin{document}
\setcounter{page}{1}

\title[Approximate bisectrix-orthogonality]{Approximately bisectrix-orthogonality preserving mappings}

\author[A. Zamani]{Ali Zamani}

\address{Department of Pure Mathematics, Center of Excellence in
Analysis on Algebraic Structures (CEAAS), Ferdowsi University of
Mashhad, P.O. Box 1159, Mashhad 91775, Iran.}
\email{\textcolor[rgb]{0.00,0.00,0.84}{zamani.ali85@yahoo.com}.}



\begin{abstract}
Regarding the geometry of a real normed space ${\mathcal X}$, we mainly introduce a notion of approximate bisectrix-orthogonality on vectors $x, y \in {\mathcal X}$ as follows:
$${x\np{\varepsilon}}_W y \mbox{~if and only if~} \sqrt{2}\frac{1-\varepsilon}{1+\varepsilon}\|x\|\,\|y\|\leq
\Big\|\,\|y\|x+\|x\|y\,\Big\|\leq\sqrt{2}\frac{1+\varepsilon}{1-\varepsilon}\|x\|\,\|y\|.$$
We study class of linear mappings preserving the approximately bisectrix-orthogonality ${\np{\varepsilon}}_W$. In particular, we show that
if $T: {\mathcal X}\to {\mathcal Y}$ is an approximate linear similarity, then
$${x\np{\delta}}_W y\Longrightarrow {Tx \np{\theta}}_W Ty \qquad (x, y\in {\mathcal X})$$
for any $\delta\in[0, 1)$ and certain $\theta\geq 0$.
\end{abstract} \maketitle

\section{Introduction and preliminaries}
There are several concepts of orthogonality appeared in the literature during the past century such as Birkhoff--James, Phythagorean, isosceles, Singer, Roberts, Diminnie, Carlsson, R\" atz, $\rho$-orthogonality, area orthogonality, etc, in an arbitrary real normed space ${\mathcal X}$, which can be regarded as generalizations of orthogonality in the inner product spaces, in general \cite{A, AMIR}. These are of intrinsic geometric interest and have been studied by many mathematicians. Among them we recall the following ones:

(i) {\it Birkhoff--James~ $\perp_B$}: $x\perp_B y$ if $\|x\|\leq \|x+ty\|$ for all scalars $t$ (see \cite{B}).

(ii) {\it Phythagorean ~ $\perp_P$}: $x\perp_P y$ if $\|x+y\|^2 = \|x\|^2 + \|y\|^2$ (see \cite{J}).

(iii) {\it Isosceles ~ $\perp_I$}: $x\perp_I y$ if $\|x+y\| = \|x-y\| $ (see \cite{J, M.M}).

(iv) {\it Roberts ~ $\perp_R$}: $x\perp_R y$ if $\|x+ty\| = \|x-ty\|$ for all scalars $t$ (see \cite{B.D}).\\
The following mapping $\langle.|.\rangle_g\  : {\mathcal X}\times{\mathcal X}\longrightarrow\mathbb{R}$ was introduced by Mili\v{c}i\v{c} \cite{MILI}:
$${\langle y|x\rangle}_g=\frac{1}{2}(\rho^\prime_+(x,y) +\rho^\prime_-(x,y)),$$
where mappings $\rho^\prime_+ , \rho^\prime_-: {\mathcal X}\times{\mathcal X}\longrightarrow\mathbb{R}$ are defined by
$$\rho^\prime_\pm(x,y) =\lim_{t\rightarrow0^\pm} \frac{\|x+ty\|^2-\|x\|^2}{2t}.$$
In addition the $\rho$-orthogonality $x\perp_\rho y$ means ${\langle y|x\rangle}_g=0$.\\
Note that $\perp_R, \perp_\rho \subseteq \perp_B$ \cite{A} and the relations $\perp_P, \perp_I, \perp_B$ are, however, independent.
If $({\mathcal H}, \langle.|.\rangle)$ is a real inner product space, then all above relations coincide with usual orthogonality $\perp$ derived from $\langle.|.\rangle$ \cite{AMIR}.\\
In the present note, we consider the so-called bisectrix-orthogonality and we study the orthogonality preserving property of this kind of orthogonality.\\
Let ${\mathcal X}$ be a real normed space and $x , y\in {\mathcal X}$. The bisectrix-orthogonality relation $x\perp_W y$ (cf. Section 5.2 in \cite{A}) is defined by
$$\Big\|\,\|y\|x+\|x\|y\,\Big\|=\sqrt{2}\|x\|\,\|y\|,$$
which for nonzero $x$ and $y$ means
$$\left \|\frac{x}{\|x\|}+\frac{y}{\|y\|}\right\|=\sqrt{2}.$$
For instance consider the space $(\mathbb{R}^2 ,\||.\||)$ where $\||(r, s)\||=\max\{|r|, |s|\}$ for $(r, s)\in\mathbb{R}^2$. Then $(1,0)\perp_W(r, s)$ if and only if either $(r, s)$ is the zero vector or
$$(r, s)\in(\{-\sqrt{2}-1,\sqrt{2}-1\})\times[-\sqrt{2},\sqrt{2}]\cup([-\sqrt{2}-1,\sqrt{2}-1]\times\{-\sqrt{2},\sqrt{2}\}).$$
Now we recall  some properties of bisectrix-orthogonality (the proofs can be found in \cite[Proposition 5.2.1]{A}):

(P.1) $x\perp_W 0, 0\perp_W y$ for all $x , y\in {\mathcal X}$;

(P.2) $x\perp_W y$ if and only if $y\perp_W x$;

(P.3) If $x\perp_W y$ and $x,y\neq0 $, then $x,y$ are linearly independent;

(P.4) If $x\perp_W y$ and $\alpha\beta\geq0 $, then $\alpha x\perp_W \beta y$;

(P.5) In an inner product space, $x\perp_W y$ if and only if ${\langle x|y\rangle}=0$.\\
By the definition of bisectrix-orthogonality and Pythagorean orthogonality one can easily get the following properties:

(P.6) For all nonzero vectors $x , y\in {\mathcal X}$, $x\perp_W y$ if and only if $\frac{x}{\|x\|}\perp_P \frac{y}{\|y\|}$;

(P.7) For all $x , y\in S_{\mathcal X}=\{z\in{\mathcal X}:\|z\|=1\}, x\perp_W y$  if and only if $x\perp_P y$.\\
We state some relations between bisectrix-orthogonality and other orthogonalities. It is known \cite{A,AMIR} that each of the following properties
implies that the norm $\|.\|$ comes from an inner product.

(P.8) $\perp_W \subseteq \perp_I$ over ${\mathcal X}$;

(P.9) $\perp_I \subseteq \perp_W$ over ${\mathcal X}$;

(P.10) $\perp_W \subseteq \perp_B$ over $S_{\mathcal X}$;

(P.11) $\perp_B \subseteq \perp_W$ over $S_{\mathcal X}$;

(P.12) If for all $x , y\in {\mathcal X}, x\perp_W y$ implies $\|x+y\|^2+\|x-y\|^2\sim2\|x\|^2+2\|y\|^2$\\
where $\sim$ stands either for $\leq$ or $\geq$.
An easy consequence of $\perp_R,\perp_\rho \subseteq \perp_B$ states that if for all $x, y\in S_{\mathcal X}$, the relation $x\perp_W y$ implies  $x\perp_R y$ or $x\perp_\rho y$, then the norm $\|.\|$ comes from an inner product.\\
In this paper we introduce two notions of approximate bisectrix-orthogonality $\np{\varepsilon}_W$ and $\perp_W^\varepsilon$ in a real normed space $\mathcal{X}$ and study the class of linear mappings which preserve the approximately bisectrix-orthogonality of type $\np{\varepsilon}_W$. In particular, we show that
if $T: {\mathcal X}\to {\mathcal Y}$ is an approximate linear similarity, then
$${x\np{\delta}}_W y\Longrightarrow {Tx \np{\theta}}_W Ty \qquad (x, y\in {\mathcal X})$$
for any $\delta\in[0, 1)$ and certain $\theta\geq 0$.

\section{Approximately bisectrix-orthogonality preserving mappings}
Let $\zeta, \eta$ be elements of an inner product space $({\mathcal H}, \langle.|.\rangle)$ and $\varepsilon\in[0 , 1)$. The approximate orthogonality $\zeta\perp^{\varepsilon}\eta$ defined by
$$|{\langle \zeta|\eta\rangle}|\leq\varepsilon\|\zeta\|\,\|\eta\|$$
or equivalently,
$$|\cos(\zeta,\eta)|\leq\varepsilon \,\,\,\,(\zeta, \eta \neq0).$$
So, it is natural to generalize the notion of approximate orthogonality for a real normed space ${\mathcal X}$.
This fact motivated Chmieli\'{n}ski and W\'{o}jcik \cite{C.W.1} to give for two elements $x , y\in {\mathcal X}$ the following definition of the approximate isosceles-orthogonality $x\np{\varepsilon}_I y$ as follows:
$$\Big|\,\|x + y\| - \|x - y\|\,\Big| \leq \varepsilon (\|x + y\| + \|x - y\|).$$
They also introduced another approximate isosceles-orthogonality $x\perp_I^\varepsilon y$ by
$$\Big|\,\|x + y\|^2 - \|x - y\|^2\,\Big| \leq 4\varepsilon \|x\|\|y\|.$$
Inspired by the above approximate isosceles-orthogonality, we propose two definitions of approximate bisectrix-orthogonality.\\
Let $\varepsilon\in[0 , 1)$ and $x , y\in {\mathcal X}$, let us put $x\np{\varepsilon}_W y$ if
$$\Big|\,\Big\|\,\|y\|x+\|x\|y\,\Big\|-\sqrt{2}\|x\|\,\|y\|\Big|\leq\varepsilon\Big(\Big\|\,\|y\|x+\|x\|y\,\Big\|+\sqrt{2}\|x\|\,\|y\|\Big)$$
or equivalently,
$$\sqrt{2}\frac{1-\varepsilon}{1+\varepsilon}\|x\|\,\|y\|\leq
\Big\|\,\|y\|x+\|x\|y\,\Big\|\leq\sqrt{2}\frac{1+\varepsilon}{1-\varepsilon}\|x\|\,\|y\|,$$
which means
$$\sqrt{2}\frac{1-\varepsilon}{1+\varepsilon}\leq\Big \|\frac{x}{\|x\|}+\frac{y}{\|y\|}\Big\|\leq\sqrt{2}\frac{1+\varepsilon}{1-\varepsilon}$$
for nonzero vectors $x$ and $y$.\\
Also we define $x\perp_W^\varepsilon y$ if
$$\Big|\,\Big\|\,\|y\|x+\|x\|y\,\Big\|^2-2\|x\|^2\,\|y\|^2\,\Big|\leq2\varepsilon\|x\|^2\,\|y\|^2$$
or equivalently,
$$\sqrt{2(1-\varepsilon)}\|x\|\,\|y\|\leq \Big\|\,\|y\|x+\|x\|y\,\Big\|\leq\sqrt{2(1+\varepsilon)}\|x\|\,\|y\|,$$
which means
$$\sqrt{2(1-\varepsilon)}\leq \Big\|\,\frac{x}{\|x\|}+\frac{y}{\|y\|}\,\Big\|\leq\sqrt{2(1+\varepsilon)}$$
for nonzero $x , y\in {\mathcal X}$.\\
It is easy to check that in the case where the norm comes from a real-valued inner product we have
$${x\np{\varepsilon}}_W y\Leftrightarrow|\langle x, y \rangle|\leq\frac{4\varepsilon}{(1-\varepsilon)^2}\|x\|\,\|y\|$$
and
$$x\perp_W^\varepsilon y\Leftrightarrow|\langle x, y \rangle|\leq \varepsilon\|x\|\,\|y\|\Leftrightarrow x\perp^\varepsilon y$$
Thus the second approximate bisectrix-orthogonality coincides with the natural notion of approximate orthogonality for inner product spaces.\\
Note that the relations $\np{\varepsilon}_W$ and $\perp_W^\varepsilon$ are symmetric and almost homogeneous in the sense that\\
$${x\np{\varepsilon}}_W y\Longrightarrow {y\np{\varepsilon}}_W x \ \ \mbox{and} \ \ {\alpha x\np{\varepsilon}}_W \beta y \ \ \mbox{for} \ \ \alpha\beta\geq0 $$
and
$$x\perp_W^\varepsilon y \Longrightarrow y\perp_W^\varepsilon x \ \ \mbox{and} \ \ \alpha x\perp_W^\varepsilon\beta y \ \ \mbox{for} \ \ \alpha\beta\geq0.$$
\begin{remark}\label{pr.0}
It is easy to see that $\perp_W^\varepsilon $ implies $\np{\varepsilon}_W $ with the same $\varepsilon$.
Also if $\varepsilon\in[0 , \frac{1}{16})$, then $\np{\varepsilon}_W $ implies $\perp_W^{16\varepsilon} $. Indeed, for $x,y\neq0$, since $0\leq\varepsilon<\frac{1}{16}$, so $\frac{1+\varepsilon}{1-\varepsilon}\leq\sqrt{1+16\varepsilon}$ and $\sqrt{1-16\varepsilon}\leq\frac{1-\varepsilon}{1+\varepsilon}$. Hence $x\np{\varepsilon}_W y$ implies $$\sqrt{2(1-16\varepsilon)}\leq\left\|\frac{x}{\|x\|}+\frac{y}{\|y\|}\right\|\leq\sqrt{2(1+16\varepsilon)}$$
or equivalently, $x\perp_W^{16\varepsilon} y$.
\end{remark}
Now, suppose that ${\mathcal X}$ and ${\mathcal Y}$ are real normed spaces of dimensions greater than or equal to two and let $\delta, \varepsilon\in[0, 1)$. We say that a linear mapping $T: {\mathcal X}\to {\mathcal Y}$ preserves the approximate bisectrix-orthogonality if
$${x\np{\delta}}_W y\Longrightarrow {Tx\np{\varepsilon}}_W Ty \qquad ( x, y\in{\mathcal X}).$$
Notice that if $\delta=\varepsilon=0$, we have
$$x\perp_W y\Longrightarrow Tx\perp_W Ty \qquad ( x, y\in{\mathcal X}),$$
and we say that $T$ preserves the bisectrix-orthogonality.\\
Koldobsky \cite{KOL} (for real spaces) and Blanco and Turn\v{s}ek \cite{B.T} (for real and complex ones) proved that a linear mapping $T: {\mathcal X}\to {\mathcal Y}$ preserving the Birkhoff orthogonality has to be a similarity, i.e., a non-zero-scalar multiple of an isometry. Further, Chmieli\'{n}ski and W\'{o}jcik \cite{C.W.2,W} proved that a linear mapping $T: {\mathcal X}\to {\mathcal Y}$ preserving the $\rho$-orthogonality has to be a similarity. Approximately orthogonality preserving mappings in the framework of normed spaces have been recently studied. In the case where $\delta=0$, Moj\v{s}kerc and Turn\v{s}ek \cite{M.T} and Chmieli\'{n}ski \cite{J.C.2} verified the properties of mappings that preserve approximate Birkhoff orthogonality. Also Chmieli\'{n}ski and W\'{o}jcik \cite{C.W.1,C.W.2} studied some properties of mappings that preserve approximate isosceles-orthogonality and $\rho$-orthogonality in the case when $\delta=0$. Recently Zamani and Moslehian \cite{ZAM} studied approximate Roberts orthogonality preserving mappings.\\
The next lemma plays an essential role in our work. It provides indeed a reverse of the triangle inequality; see \cite{M.D}.
\begin{lemma}\cite[Theorem 1]{MAL}\label{app.2}
Let ${\mathcal X}$ be a normed space and $x, y\in {\mathcal X}\setminus\{0\}$. Then
\begin{align*}
\|x\|+\|y\|+(\Big\|\,\frac{x}{\|x\|}&+\frac{y}{\|y\|}\,\Big\|-2)\max\{\|x\|, \|y\|\}
\\&\leq\|x+y\|
\\&\leq\|x\|+\|y\|+(\Big\|\,\frac{x}{\|x\|}+\frac{y}{\|y\|}\,\Big\|-2)\min\{\|x\|, \|y\|\}
\end{align*}
\end{lemma}
To reach our main result, we need some lemmas, which are interesting on their own right. We state some prerequisites for the first lemma. For a bounded linear mapping $T: {\mathcal X}\to {\mathcal Y}$, let $\|T\|=\sup\{ \|Tx\|; \ \|x\|=1\}$ denote the operator norm and $[T]: =\inf\{ \|Tx\|; \ \|x\|=1\}$. Notice that for any $x\in {\mathcal X}$, we have $[T]\|x\|\leq\|Tx\|\leq\|T\|\|x\|.$
\begin{lemma}\label{th.1}
Let $\delta, \varepsilon\in[0, 1)$. If a nonzero bounded linear mapping $T: {\mathcal X}\to {\mathcal Y}$ satisfies $$\frac{1-\varepsilon}{1+\varepsilon}\gamma\|x\|\leq\|Tx\|\leq\frac{1+\varepsilon}{1-\varepsilon}\gamma\|x\|$$
for all $x\in{\mathcal X}$ and all $\gamma\in\Big[\frac{1-\delta}{1+\delta}[T] , \frac{1+\delta}{1-\delta}\|T\|\Big]$, then
$${x\np{\delta}}_W y\Longrightarrow {Tx \np{\varepsilon}}_W Ty \qquad  (x , y\in {\mathcal X}).$$
\end{lemma}
\begin{proof}
Let $x, y\in {\mathcal X}\setminus\{0\}$ and $x\np{\delta}_W y$. Then $\sqrt{2}\frac{1-\delta}{1+\delta}\leq\Big \|\frac{x}{\|x\|}+\frac{y}{\|y\|}\Big\|\leq\sqrt{2}\frac{1+\delta}{1-\delta}$. If $x=sy$ for some $s\in\mathbb{R}\smallsetminus\{0\}$, then for $\gamma=\frac{1-\delta}{1+\delta}[T]$ we have
\begin{align*}
\Big\|\frac{Tx}{\|Tx\|}+\frac{Ty}{\|Ty\|}\Big\|&=\Big\|T\Big(\frac{x}{\|Tx\|}+\frac{y}{\|Ty\|}\Big)\Big\|
\\&=\Big\|T\Big(\frac{sy}{\|sTy\|}+\frac{y}{\|Ty\|}\Big)\Big\|
\\&=\frac{\|y\|}{\|Ty\|}\Big\|T\Big(\frac{sy}{\|sy\|}+\frac{y}{\|y\|}\Big)\Big\|
\\&=\frac{\|y\|}{\|Ty\|}\Big\|T\Big(\frac{x}{\|x\|}+\frac{y}{\|y\|}\Big)\Big\|
\\&\leq\frac{\|y\|}{\|Ty\|}\frac{1+\varepsilon}{1-\varepsilon}\gamma\Big\|\frac{x}{\|x\|}+\frac{y}{\|y\|}\Big\|
\\&= \frac{\|y\|}{\|Ty\|}\frac{1+\varepsilon}{1-\varepsilon}\frac{1-\delta}{1+\delta}[T]\Big\|\frac{x}{\|x\|}+\frac{y}{\|y\|}\Big\|
\\&\leq \frac{1+\varepsilon}{1-\varepsilon}\frac{1-\delta}{1+\delta}\sqrt{2}\frac{1+\delta}{1-\delta}
\\&=\sqrt{2}\frac{1+\varepsilon}{1-\varepsilon}\,,
\end{align*}
whence $\Big\|\frac{Tx}{\|Tx\|}+\frac{Ty}{\|Ty\|}\Big\|\leq\sqrt{2}\frac{1+\varepsilon}{1-\varepsilon}$. Similarly, $\sqrt{2}\frac{1-\varepsilon}{1+\varepsilon}\leq\Big\|\frac{Tx}{\|Tx\|}+\frac{Ty}{\|Ty\|}\Big\|$. Thus $Tx\np{\varepsilon}_W Ty$.
Assume that $x, y$ are linearly independent. Set $\gamma_0: =\frac{\sqrt{2}}{\left\|\frac{x}{\|Tx\|}+\frac{y}{\|Ty\|}\right\|}$. We may assume that $\frac{\|x\|}{\|Tx\|}\leq\frac{\|y\|}{\|Ty\|}.$ By Lemma \ref{app.2} we have
\begin{align*}
\Big\|\frac{x}{\|Tx\|}+\frac{y}{\|Ty\|}\Big\|&\leq\frac{\|x\|}{\|Tx\|}+\frac{\|y\|}{\|Ty\|}+(\Big\|\,\frac{x}{\|x\|}+
\frac{y}{\|y\|}\,\Big\|-2)\min\{\frac{\|x\|}{\|Tx\|}, \frac{\|y\|}{\|Ty\|}\}
\\&\leq\frac{\|y\|}{\|Ty\|}+(\sqrt{2}\frac{1+\delta}{1-\delta}-1)\frac{\|x\|}{\|Tx\|}
\\&\leq\frac{1}{[T]}+(\sqrt{2}\frac{1+\delta}{1-\delta}-1)\frac{1}{[T]}
\\&=\sqrt{2}\frac{1+\delta}{1-\delta}\frac{1}{[T]}.
\end{align*}
So that $\gamma_0\geq\frac{\sqrt{2}}{\sqrt{2}\frac{1+\delta}{1-\delta}\frac{1}{[T]}}=\frac{1-\delta}{1+\delta}[T]$.\\
Similarly we get $\gamma_0\leq \frac{1+\delta}{1-\delta}\|T\|$. Thus $\gamma_0\in\Big[\frac{1-\delta}{1+\delta}[T] , \frac{1+\delta}{1-\delta}\|T\|\Big]$. Our hypothesis implies that
$$\frac{1-\varepsilon}{1+\varepsilon}\gamma_0 \|z\|\leq\|Tz\|\leq\frac{1+\varepsilon}{1-\varepsilon}\gamma_0 \|z\| \qquad (z\in{\mathcal X})$$
or equivalently,
$$\Big|\|Tz\|-\gamma_ 0\|z\|\Big|\leq\varepsilon(\|Tz\|+\gamma_0\|z\|) \qquad (z\in{\mathcal X})$$
Putting $\|Ty\|x+\|Tx\|y$ instead of $z$ in the above inequality we get
\begin{align*}
\Big|\,\Big\|\,\|Ty\|Tx+&\|Tx\|Ty\,\Big\|-\frac{\sqrt{2}}{\left\|\frac{x}{\|Tx\|}+\frac{y}{\|Ty\|}\right\|}\Big\|\,\|Ty\|x+\|Tx\|y\,\Big\|\,\Big|
\\&\leq\varepsilon\Big(\Big\|\,\|Ty\|Tx+\|Tx\|Ty\,\Big\|+\frac{\sqrt{2}}{\left\|\frac{x}{\|Tx\|}+\frac{y}{\|Ty\|}\right\|}\Big\|\,\|Ty\|x+\|Tx\|y\,\Big\|\Big).
\end{align*}
Thus
$$\Big|\,\Big\|\,\|Ty\|Tx+\|Tx\|Ty\,\Big\|-\sqrt{2}\|Tx\|\,\|Ty\|\,\Big|\leq\varepsilon\Big(\Big\|\,\|Ty\|Tx+\|Tx\|Ty\,\Big\|+\sqrt{2}\|Tx\|\,\|Ty\|\Big),$$
whence $Tx \np{\varepsilon}_W Ty$.
\end{proof}
\begin{lemma}\label{co.2}
Let $\delta, \varepsilon\in[0, 1)$. If a nonzero bounded linear mapping $T: {\mathcal X}\to{\mathcal Y}$ satisfies $\frac{1+\delta}{1-\delta}\|Tz\|\,\|u\|\leq\frac{1+\varepsilon}{1-\varepsilon}\|Tu\|\,\|z\|$ for all $z, u\in{\mathcal X}$, then
$${x\np{\delta}}_W y\Longrightarrow {Tx \np{\varepsilon}}_W Ty \qquad  (x , y\in {\mathcal X}).$$
\end{lemma}
\begin{proof}
By our assumption we have, $\frac{1+\delta}{1-\delta}\|Tz\|\leq\frac{1+\varepsilon}{1-\varepsilon}\|Tu\|$ for all $z, u$ with $\|z\|=\|u\|=1.$
Passing to the infimum over $\|u\|=1$, we get
$$\frac{1+\delta}{1-\delta}\|Tz\|\leq\frac{1+\varepsilon}{1-\varepsilon}[T] \quad (\|z\|=1),$$
and passing to the supremum over $\|z\|=1$ we obtain
$$\frac{1+\delta}{1-\delta}\|T\|\leq\frac{1+\varepsilon}{1-\varepsilon}[T].$$
Now, let $\gamma\in\Big[\frac{1-\delta}{1+\delta}[T] , \frac{1+\delta}{1-\delta}\|T\|\Big]$ and $x\in{\mathcal X}$. Therefore we have
\begin{align*}
\frac{1-\varepsilon}{1+\varepsilon}\gamma\|x\|&\leq\frac{1-\varepsilon}{1+\varepsilon}\times\frac{1+\delta}{1-\delta}\|T\|\,\|x\|
\\&\leq\frac{1-\varepsilon}{1+\varepsilon}\times\frac{1+\delta}{1-\delta}\times\frac{1+\varepsilon}{1-\varepsilon}\times\frac{1-\delta}{1+\delta}[T]\,\|x\|
\\&\leq\|Tx\|
\\&\leq\|T\|\,\|x\|
\\&\leq\frac{1+\varepsilon}{1-\varepsilon}\times\frac{1-\delta}{1+\delta}[T]\,\|x\|
\\&\leq\frac{1+\varepsilon}{1-\varepsilon}\times\frac{1-\delta}{1+\delta}\times\frac{1+\delta}{1-\delta}\gamma\|x\|
\\&=\frac{1+\varepsilon}{1-\varepsilon}\gamma\|x\|.
\end{align*}
Thus
$$\frac{1-\varepsilon}{1+\varepsilon}\gamma\|x\|\leq\|Tx\|\leq\frac{1+\varepsilon}{1-\varepsilon}\gamma\|x\|$$
Making a use of Lemma \ref{th.1} just completes the proof.
\end{proof}
We are now in position to establish the main result. Following \cite{M.T}, we say that a linear mapping $U: {\mathcal X}\to {\mathcal Y}$ is an approximate linear isometry if
$$(1-\varphi_1 (\varepsilon))\|z\|\leq\|Uz\|\leq(1+\varphi_2 (\varepsilon))\|z\| \quad (z \in {\mathcal X}),$$
where $\varphi_1 (\varepsilon)\rightarrow0$ and $\varphi_2 (\varepsilon)\rightarrow0$ as $\varepsilon\rightarrow0$. Notice that if $\varphi_1 (\varepsilon)=\varphi_2 (\varepsilon)=0$, then $U$ is an isometry.

A linear mapping $U: {\mathcal X}\to{\mathcal Y}$ is said to be an approximate similarity if it is a non-zero-scalar multiple of an approximate linear isometry, or equivalently it satisfies $$\lambda(1-\varphi_1 (\varepsilon))\|w\|\leq\|Uw\|\leq\lambda(1+\varphi_2 (\varepsilon))\|w\|$$ for some unitary $U$, some $\lambda > 0$ and for all $w \in {\mathcal X}$, where
$\varphi_1 (\varepsilon)\rightarrow0$ and $\varphi_2 (\varepsilon)\rightarrow0$ as $\varepsilon\rightarrow0$.

\begin{theorem}
Let $U: {\mathcal X}\to{\mathcal Y}$ be an approximate linear similarity and $\delta\in[0, 1)$. If a nonzero bounded linear
mapping $T: {\mathcal X}\to{\mathcal Y}$ satisfies $\|T-U\|\leq\varepsilon\|U\|$, then
$${x\np{\delta}}_W y\Longrightarrow {Tx \np{\theta}}_W Ty \qquad  (x , y\in {\mathcal X}),$$
where $\theta=\frac{2\delta+2\varepsilon+(1-\delta)\varphi_1 (\varepsilon)+(1+\delta+2\varepsilon)\varphi_2 (\varepsilon)}{2+2\delta\varepsilon-(1-\delta)\varphi_1 (\varepsilon)+(1+\delta+2\delta\varepsilon)\varphi_2 (\varepsilon)}$.
\end{theorem}
\begin{proof}
For any $w\in {\mathcal X}$ we have
$$\Big|\,\|Tw\|-\|Uw\|\,\Big|\leq\|Tw-Uw\|\leq\|T-U\|\,\|w\|\leq\varepsilon\|U\|\,\|w\|\leq\varepsilon\lambda(1+\varphi_2 (\varepsilon))\|w\|,$$
whence
$$-\varepsilon\lambda(1+\varphi_2 (\varepsilon))\|w\|\leq\|Tw\|-\|Uw\|\leq\varepsilon\lambda(1+\varphi_2 (\varepsilon))\|w\|.$$
Since
$$\lambda(1-\varphi_1 (\varepsilon))\|w\|\leq\|Uw\|\leq\lambda(1+\varphi_2 (\varepsilon))\|w\|,$$
therefore we get
$$\lambda\Big[(1-\varphi_1 (\varepsilon))-\varepsilon(1+\varphi_2 (\varepsilon))\Big]\|w\|\leq\|Tw\|\leq\lambda(1+\varepsilon)(1+\varphi_2 (\varepsilon))\|w\|.$$
Thus for any $z, u\in{\mathcal X}$, we have
\begin{align*}
\frac{1+\delta}{1-\delta}\|Tz\|\,\|u\|&\leq\frac{1+\delta}{1-\delta}\lambda(1+\varepsilon)(1+\varphi_2 (\varepsilon))\|z\|\,\frac{\|Tu\|}{\lambda\Big[(1-\varphi_1 (\varepsilon))-\varepsilon(1+\varphi_2 (\varepsilon))\Big]}
\\&=\frac{(1+\varepsilon)(1+\varphi_2 (\varepsilon))(1+\delta)}{[(1-\varphi_1 (\varepsilon))-\varepsilon(1+\varphi_2 (\varepsilon))](1-\delta)}\|Tu\|\|z\|
\\&=\frac{1+\frac{2\delta+2\varepsilon+(1-\delta)\varphi_1 (\varepsilon)+(1+\delta+2\varepsilon)\varphi_2 (\varepsilon)}{2+2\delta\varepsilon-(1-\delta)\varphi_1 (\varepsilon)+(1+\delta+2\delta\varepsilon)\varphi_2 (\varepsilon)}}{1-\frac{2\delta+2\varepsilon+(1-\delta)\varphi_1 (\varepsilon)+(1+\delta+2\varepsilon)\varphi_2 (\varepsilon)}{2+2\delta\varepsilon-(1-\delta)\varphi_1 (\varepsilon)+(1+\delta+2\delta\varepsilon)\varphi_2 (\varepsilon)}}\|Tu\|\|z\|
\\&=\frac{1+\theta}{1-\theta}\|Tu\|\|z\|.
\end{align*}
Therefore $\frac{1+\delta}{1-\delta}\|Tz\|\,\|u\|\leq\frac{1+\theta}{1-\theta}\|Tu\|\,\|z\|$. Now the assertion follows from Lemma \ref{co.2}.
\end{proof}

As a consequence, with $\varepsilon=0$ and $T=U$, we have
\begin{corollary}
Let $T: {\mathcal X}\to {\mathcal Y}$ be an approximate linear similarity. Then
$${x\np{\delta}}_W y\Longrightarrow {Tx \np{\theta}}_W Ty \qquad (x, y\in {\mathcal X})$$
for any $\delta\in[0, 1)$, where $\theta=\frac{2\delta+(1-\delta)\varphi_1 (\varepsilon)+(1+\delta)\varphi_2 (\varepsilon)}{2-(1-\delta)\varphi_1 (\varepsilon)+(1+\delta)\varphi_2 (\varepsilon)}$.
\end{corollary}

\textbf{Acknowledgement.} The author would like to thank his PhD supervisor, Prof. M. S. Moslehian, for his useful comments improving the paper.

\bibliographystyle{amsplain}

\end{document}